\documentclass[11pt,a4paper,oneside,reqno,final]{amsart}

\usepackage[T1]{fontenc}
\usepackage[utf8]{inputenc}
\usepackage{lmodern}
\usepackage{microtype}

\usepackage{amsmath,amssymb,amsthm,mathtools,mathrsfs}
\usepackage[dvipsnames]{xcolor}
\usepackage{graphicx}
\usepackage{tikz}
\usepackage{caption,subcaption}
\usepackage{float,placeins}
\usepackage{comment}
\usepackage[shortlabels]{enumitem}
\usepackage{setspace}

\IfFileExists{bbm.sty}{\usepackage{bbm}}{}
\IfFileExists{dsfont.sty}{\usepackage{dsfont}}{\let\mathds\mathbf}

\usepackage{hyperref}
\hypersetup{
  colorlinks=true,
  linkcolor=blue,
  citecolor=teal,
  urlcolor=black,
  pdftitle={Anchored isoperimetry in percolation},
  pdfauthor={Sahar Diskin, Lyuben Lichev, Ritvik Ramanan Radhakrishnan, Benny Sudakov, Laury Verhoeven}
}
\usepackage[capitalise]{cleveref}

\allowdisplaybreaks

\setlist{itemsep=2pt,parsep=1pt,topsep=2pt,partopsep=0pt}
\setlist[enumerate]{leftmargin=*,labelindent=\parindent}

\theoremstyle{plain}
\newtheorem{theorem}{Theorem}
\newtheorem{thm}{Theorem}[section]

\newtheorem{lemma}[thm]{Lemma}
\newtheorem{proposition}[thm]{Proposition}

\theoremstyle{definition}

\newtheorem{algorithm}{Algorithm}

\newtheorem{remark}[thm]{Remark}

\makeatletter
\def\namedlabel#1#2{\begingroup
  \def\@currentlabel{#2}%
  \label{#1}\endgroup
}

\newcommand*{\barfix}[2][.175ex]{%
  \mathpalette{\@barfix{#1}}{#2}%
}
\newcommand*{\@barfix}[3]{%
  \vbox{%
    \kern#1\relax
    \hbox{$#2#3\m@th$}%
  }%
}
\makeatother

\newcommand{\cB}{\mathcal{B}}

\newcommand{\cE}{\mathcal{E}}
\newcommand{\cF}{\mathcal{F}}

\newcommand{\bC}{\mathbf{C}}
\newcommand{\bF}{\mathbf{F}}
\newcommand{\bH}{\mathbf{H}}
\newcommand{\bT}{\mathbf{T}}
\newcommand{\bw}{\mathbf{w}}
\newcommand{\br}{\mathbf{r}}

\usepackage[margin=0.85in]{geometry}

\def\eps{\varepsilon}

\title{Anchored isoperimetry in percolation}

\author[Diskin, Lichev, Radhakrishnan, Sudakov, Verhoeven]{%
\mbox{Sahar Diskin\textsuperscript{1}}\quad
\mbox{Lyuben Lichev\textsuperscript{2}}\quad
\mbox{Ritvik Ramanan Radhakrishnan\textsuperscript{1}}\\
\mbox{Benny Sudakov\textsuperscript{1}}\quad
\mbox{Laury Verhoeven\textsuperscript{1}}}
\date{}

\makeatletter
\g@addto@macro\@setauthors{%
  \begingroup
  \centering\normalfont\footnotesize
  \setstretch{1.0}\setlength{\parskip}{0pt}%
  \textsuperscript{1}Department of Mathematics, ETH Z\"urich,
  R\"amistrasse 101, 8092 Z\"urich, Switzerland\par
  \textsuperscript{2}Institute of Statistics and Mathematical Methods in Economics,\par
  TU Wien, A-1040 Vienna, Austria\par
  \endgroup
}
\makeatother

\begin{document}

\begin{abstract}
We show that, for any infinite transitive graph $G$, the infinite cluster after supercritical percolation almost surely inherits the anchored isoperimetry of $G$. Moreover, the probability of violating the corresponding isoperimetric inequality decays exponentially in the isoperimetric profile of the host graph. 
This was previously known only for nonamenable transitive graphs, $\mathbb{Z}^d$, and very recently for transitive graphs of polynomial growth. Our proof combines the thinning-and-sprinkling technique with a novel counting algorithm.
\end{abstract}

\maketitle

\section{Introduction}

Percolation provides a natural way to probe the geometry of an infinite graph $G$: randomly delete its edges and ask which structural properties survive. In the supercritical regime, the resulting random graph contains infinite connected components. A fundamental question is to what extent these clusters retain the large-scale geometric properties of the original graph. One particularly important aspect of this question is isoperimetry, which measures how the boundary of a vertex set grows with its size.

More precisely, let $G=(V,E)$ be an infinite, connected, locally finite, vertex-transitive graph of degree $d$. In Bernoulli bond percolation, each edge of $G$ is retained independently with probability $p\in[0,1]$. This produces a random subgraph $G_p$ with distribution $\mathbb P_p$, whose connected components are called \emph{clusters}. The critical probability is defined as
\[
p_c=p_c(G)\coloneqq
\inf\{p\in[0,1]\colon G_p
\text{ contains an infinite cluster almost surely}\}.
\]
The regimes $p<p_c$ and $p>p_c$ are called \emph{subcritical} and \emph{supercritical}, respectively. Percolation was introduced in 1957 by Broadbent and Hammersley~\cite{BH57}, and was initially studied mainly on the hypercubic lattice $\mathbb Z^d$. An important step towards a more general understanding was made by Benjamini and Schramm~\cite{perco_beyond_zd}, who initiated the systematic study of percolation on arbitrary infinite transitive graphs in 1996.

An important quantitative measure of the global geometry of a graph is its \emph{isoperimetric profile}. For a finite set $S\subseteq V$, let $\partial_G S$ denote the set of edges of $G$ with exactly one endpoint in $S$. The isoperimetric profile of $G$ is defined by
\[
\Phi_G(n)\coloneqq
\inf\{|\partial_G W|\colon
W\subseteq V,\ n\le |W|<\infty\}.
\]
For example, $\Phi_G(n)=\Theta(n)$ when $G$ is nonamenable, whereas
$\Phi_{\mathbb Z^d}(n)=\Theta(n^{(d-1)/d})$
for $d\ge 2$.
In the study of percolation, one might naturally hope that infinite supercritical clusters inherit the isoperimetric profile of $G$, up to multiplicative constants. However, such inequalities cannot hold uniformly over all finite subsets of a cluster. Indeed, when $p\in(p_c,1)$, infinite clusters contain arbitrarily long bare paths, and taking a long subpath of their internal vertices yields a large connected set with only two boundary edges (see, e.g., \cite{HH21,H22}).

The natural way to overcome this difficulty is to consider an \emph{anchored} version of isoperimetry, where we restrict attention to connected sets containing a fixed vertex. Given an infinite connected graph $H$ and a vertex $o\in V(H)$, its anchored isoperimetric profile is defined by
\[
\Phi_{H,o}(n)\coloneqq
\inf\{|\partial_H W|\colon
o\in W\subseteq V(H),\
W\text{ is connected},\
n\le |W|<\infty\}.
\]
Let $C_o$ denote the cluster containing $o$ in $G_p$. When $p>p_c$ and $C_o$ is infinite, it is natural to ask whether its anchored isoperimetric profile is comparable to that of $G$. More precisely, for infinite connected graphs $H\subseteq G$ containing $o$, we say that $H$ has \emph{anchored isoperimetry with respect to $G$} if
\[
\liminf_{n\to\infty}
\frac{\Phi_{H,o}(n)}{\Phi_G(n)}>0.
\]
We say that $G$ has \emph{anchored isoperimetry} if, for every $p>p_c(G)$, the cluster $C_o$ almost surely has anchored isoperimetry with respect to $G$, conditionally on being infinite. When $G$ is nonamenable, we have $\Phi_G(n)=\Theta(n)$, and anchored isoperimetry of $H$ with respect to $G$ simply means that there exists $c>0$ such that every sufficiently large finite connected set $W\subseteq V(H)$ containing $o$ satisfies $|\partial_H W|\ge c|W|$. This property is called \emph{anchored expansion}. Beyond their geometric significance, these notions have important consequences for random walks and phase transitions. In particular, anchored expansion implies positive speed of the simple random walk on bounded-degree graphs~\cite{Vir_g_2000}; see also~\cite{delmotte2015exittimeanchoredexpansion,HH21,HSS00}.

The notion of anchored expansion was introduced by Benjamini, Lyons, and Schramm~\cite{BLS99}, who asked whether it is preserved under supercritical percolation. Chen and Peres~\cite{Peres}, with an appendix by Pete, established anchored expansion for percolation parameters sufficiently close to one. Subsequently, Pete~\cite{Pete08}, using ideas from that appendix, proved the corresponding anchored isoperimetric inequality for supercritical percolation on $\mathbb Z^d$.
The question for nonamenable transitive graphs was eventually resolved for every $p>p_c$ by Hermon and Hutchcroft~\cite{HH21}, who also formulated the general anchored isoperimetry problem for arbitrary transitive graphs (Conjecture~5.4 in~\cite{HH21}). More recently, Martineau and Panagiotis~\cite{MP26} established anchored isoperimetry for transitive graphs of polynomial growth.

Our main result settles this problem for all infinite connected transitive graphs. Moreover, it provides a quantitative strengthening: the probability that an infinite cluster violates the desired anchored isoperimetric inequality at scale $n$ decays exponentially in $\Phi_G(n)$.

\begin{theorem}\label{thm: main}
Fix any infinite connected transitive graph $G$ with a vertex $o\in V(G)$. For every $p>p_c(G)$, there is a constant $c=c(G,p)>0$ such that, for every $n\ge 1$,
\[
\mathbb{P}_p\bigl(
|C_o|=\infty,\,
\Phi_{C_o,o}(n)<c\cdot\Phi_G(n)
\bigr)
\le \exp\bigl(-c\cdot\Phi_G(n)\bigr).
\]
\end{theorem}

Previous approaches to anchored isoperimetry have relied on
\emph{supercritical sharpness}, which provides an upper bound
on the probability that a vertex belongs to a large but finite
cluster. Very recently, supercritical sharpness was established for all infinite transitive graphs~\cite{DERST26}. More precisely, it was shown that, for every
$p>p_c(G)$, there exists $c=c(G,p)>0$ such that
\[
\mathbb{P}_p(n\le |C_o|<\infty)
\le \exp(-c\Phi_G(n)),
\]
for every $n\ge 1$.
This result naturally suggests deriving anchored isoperimetry
directly from the decay of finite cluster sizes, as in the
previously known cases~\cite{Pete08,HH21,MP26}.
However, the general approach to converting such tail bounds
into anchored isoperimetry incurs a logarithmic loss in $\Phi_G$
(see~\cite[Proposition~3.1]{H22}), and therefore does not
yield the desired bound for arbitrary transitive graphs.

Our proof takes a different route. It combines the
thinning-and-sprinkling method, developed in recent papers ~\cite{DK24,ADLZ25,CDLS26,CHS26}, with a novel
counting algorithm. The key idea is to count the possible
ways in which connected sets can touch infinite clusters,
rather than counting the sets themselves. This `reduction
in entropy' allows us to avoid the
logarithmic loss and obtain a short, unified proof of
anchored isoperimetry for all transitive graphs.
Moreover, the argument uses transitivity only through
a probabilistic estimate from~\cite[Proposition~2.1]{DERST26},
without requiring any detailed structural information
about the underlying graph.

The rest of this paper is organised as follows.
We present a high-level overview of the proof in
Section~\ref{s:outline}, and give the full details
in Section~\ref{s:proof}.

\section{Proof outline}\label{s:outline}

Rather than invoking supercritical sharpness directly, our proof
relies on the following key estimate from~\cite{DERST26},
which concerns the number of edges leaving a finite set
that can connect to infinity without returning to the set.

For a finite set $S\subseteq V$, let
$\Phi_G(S)\coloneqq
\min\{|\partial_G W|: S\subseteq W,\ |W|<\infty\}$
denote the minimum size of an edge cutset separating $S$
from infinity. Given $q>p_c(G)$, we also write
\[
t_q(S)\coloneqq
|\{uv\in E(G):u\in S,\,
v\xleftrightarrow{G_q\setminus S}\infty\}|.
\]
The following is a convenient consequence of
\cite[Proposition~2.1]{DERST26}.

\begin{proposition}\label{prop: touches}
Fix an infinite connected transitive graph $G$ of degree $d$.
For every $q\in(p_c(G),1]$ and every $\eps\in(0,1)$,
there exists $c>0$ such that, for every finite
$S\subseteq V$ and every positive integer
$t\le c\Phi_G(S)$,
\[
\mathbb P_q(t_q(S)\ge t)\ge (1-\eps)^t.
\]
\end{proposition}

The relation between this notion of ``touches'' (that is, edges between $S$ and the infinite clusters in $G_q$) and anchored isoperimetry goes back to Pete \cite{Pete08}. 
Roughly speaking, Pete conjectured \cite[Conjecture 12.32]{pete2014probability} that the probability under $q$-percolation that the cluster $C_o$ is finite and satisfies $t_q(C_o)\ge t$ should decay exponentially in $t$ (note that this is, morally, the opposite direction of Proposition~\ref{prop: touches} for a special random set $S$). 
This conjecture, combined with supercritical sharpness and the surgery approach of \cite[Appendix A]{Peres}, would imply Theorem~\ref{thm: main}; indeed, this was the approach taken in \cite{Pete08,MP26} for $G=\mathbb{Z}^d$ and for transitive graphs of polynomial growth. 
However, for arbitrary transitive graphs, the conjecture of Pete remains open; we take a different path instead.

Call a connected set $S\subseteq C_o$ \textit{boundary-minimiser} if there is no finite connected superset $W\supset S$ with $|\partial_{C_o}W| \le |\partial_{C_o}S|$. (Formally, we will work inside a large finite ball, and this notion will be restricted to this ball).
Note that $\Phi_G(S)$ is nondecreasing under inclusion and therefore, for Theorem~\ref{thm: main} to fail, there must be a boundary-minimiser set in $G_p$ containing $o$ with rather large boundary in $G$ but rather small boundary in $G_p$.

To analyse this event, we will use a variant of the thinning-and-sprinkling method.
Fix $q\in(p_c,p)$. Since we are interested in sets $S$ with few boundary edges in $G_p$, requiring that $S$ also has no open boundary edges in $G_q$ costs, for a fixed $p$-configuration, a factor of $\big(1-q/p\big)^{|\partial_{G_p} S|}$.
We further restrict our attention to sets with many $q$-touches; by Proposition~\ref{prop: touches} and FKG, this costs a factor of at most $(1-\eps)^{-t}$ when $t\le c\Phi_G(S)$. 
Moreover, it follows from \cite{M11} that the number of boundary-minimisers with at most $k$ boundary edges is at most $4^k$.
Our goal will thus be to show that it is very unlikely that $C_o$ contains a boundary-minimiser set $S$ (in $G_p$) with many $q$-touches, no boundary edges in $G_q$ and too few boundary edges in $G_p$. 

Fix $n\ge 1$, and consider the following setup. First, reveal the entire $q$-configuration and let $K_q$ be the union of its infinite clusters.
If $o\in K_q$, the obstruction we seek is impossible. Otherwise, reveal the $p$-edges within $G\setminus K_q$, and let $H$ be the component of $o$ in the resulting graph restricted to a large finite ball.
Let $M\subseteq V(H)\setminus\{o\}$ be the vertices with a neighbour in $K_q$. For a connected set $S\subseteq V(H)$ containing $o$, its \textit{fingerprint} $U=S\cap M$, together with the possible touches at $o$, determines all its touch edges.
We want to bound the probability that there is such a set $S$ with no $q$-open boundary edges, $t_q(S)\ge t$ and $|\partial_{G_p}S|\le k$, where $t=\Theta(\Phi_G(n))$ and $k\le ct$ for a sufficiently small constant $c$. 
(We note that the upper bound provided here will hold not only for boundary-minimiser sets, but more generally.)

A natural first attempt is to union bound over all possible sets $S$; however, this approach leads to a severe overcounting. The first key observation is that we can ``reduce the entropy'' by counting instead the \textit{possible fingerprints} $U$, rather than the sets themselves (see Figure~\ref{fig:shared-touches}).

\begin{figure} 
    \centering
    \includegraphics[width=0.7\linewidth]{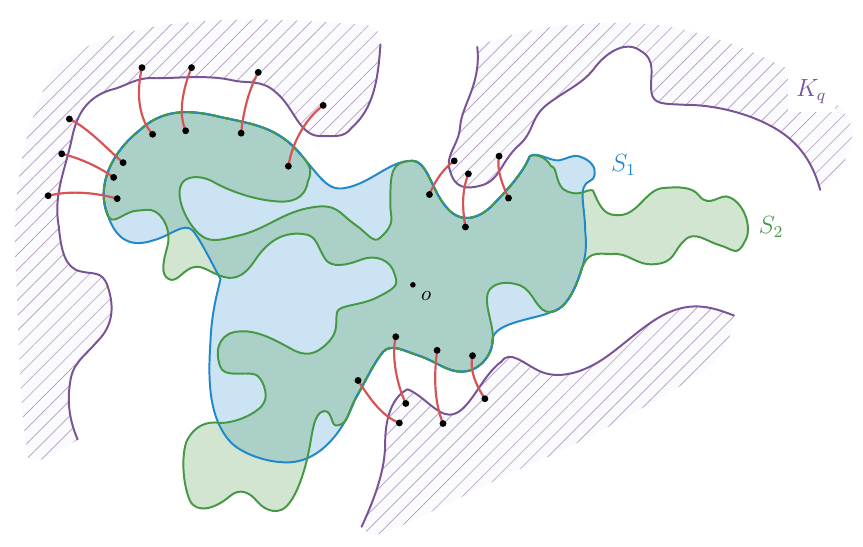}
    \caption{Two different connected sets $S_1$ (blue) and $S_2$ (green) meeting the same touches, drawn in red.}
    \label{fig:shared-touches}
\end{figure}

\FloatBarrier
 
Indeed, once the $q$-configuration has been revealed, each edge between $K_q$ and its complement becomes $p$-open independently with probability $(p-q)/(1-q)$. Thus, for every fixed choice of $t$ such edges, the probability that at most $k$ of them are in $G_p$ is 
\begin{equation}\label{eq:proba}
\mathbb{P}\left(\operatorname{Bin}\left(t,(p-q)/(1-q)\right)\le k\right)\le \exp(-c't)
\end{equation} 
for some $c'>0$ by Chernoff's bound, provided $k\ll t$.
This point justifies why we require that $t_q(S)\ge t\gg k$: while we only ``paid'' for closing $k$ edges in the thinning stage of the argument, the ``price'' for opening at most $k$ edges among $t$ upon sprinkling is much more significant.
Now, we would like to conclude by a union bound argument. Say that $U\subseteq M$ is \emph{$k$-feasible} if $U=W\cap M$ for some connected set $W\subseteq V(H)$ containing $o$ with $|\partial_H W|\le k$. 
It suffices to control the number of such fingerprints of each size $u=|U|\ge t/d$: indeed, the number of edges from $K_q$ touching $U$ lies in the interval $[u,du]$,
and the sprinkling probability \eqref{eq:proba} decays exponentially with this number.

To do so, we develop a novel counting algorithm specialised for counting sets with small boundary.
Roughly speaking, given a witness $W$ for a feasible set $U$, the algorithm grows an explored region $C$ from $o$. The region $C$ stays connected in $G_p$ and is maximised conditionally on keeping the edge cut separating it from $M\setminus V(C)$ of minimum size. Then, the algorithm examines a prescribed edge leaving that side. It writes a $1$ and enlarges the region if the other endpoint belongs to $W$; otherwise it writes a $0$ and deletes the edge. The crucial feature is that the resulting binary word determines $U=W\cap M$, though not necessarily $W$ itself (Lemma~\ref{cl: property 1}). Submodularity shows that the word has at most $k$ zeroes, and the estimate $|\partial_H W|\le k$ allows to bound its length by $3k+(d+1)|U|$ (Proposition~\ref{prop: feasible}). Counting these words gives the required entropy bound; combined with the exponential sprinkling estimate \eqref{eq:proba} and a union bound over the possible sizes $u=|U|$, this completes the argument.

\section{Anchored isoperimetry via thinning and sprinkling}\label{s:proof}

Fix $p>p_c(G)$, choose $q\in(p_c(G),p)$, and consider the natural coupling $G_q\subseteq G_p$. We write $\partial_p$ and $\partial_q$ for the corresponding edge boundaries, and use $\mathbb P$ for the joint law of the coupled configurations. Choose $\eps=\eps(p,q)>0$ sufficiently small (its precise choice will be determined by comparing the thinning and sprinkling estimates at the end of the proof).
By Proposition~\ref{prop: touches}, there exists $c_0=c_0(G,q,\eps)>0$ such that $\mathbb P_q(t_q(S)\ge t')\ge(1-\eps)^{t'}$ whenever $t'\le c_0\Phi_G(S)$. Fix $m\ge1$ and set $t=c_0m$. Let $\alpha=\alpha(G,p,q,c_0)>0$ be sufficiently small (with its precise value to be chosen later) and set $k=\alpha m$. 

Recall that $\cB_m$ is the event that there exists a $G_p$-connected set $S$ containing $o$ with $\Phi_G(S)\ge m$ and $|\partial_pS|\le k$. To work in a finite graph, fix $r\ge1$, put $B=B_r(o)$, and let $\cB_{m,r}$ be the same event with the additional restriction $S\subseteq B$. Since $\cB_{m,r}$ increases to $\cB_m$ as $r\to\infty$, it suffices to obtain a bound on $\mathbb{P}_p(\cB_{m,r})$ which is independent of $r$.

We compare $\cB_{m,r}$ with the (simpler to analyse) event $\cE_{t,k,r}$ that there exists a $G_p$-connected set $o\in S\subseteq B$ satisfying $|\partial_pS|\le k$, $\partial_qS=\varnothing$, and $t_q(S)\ge t$. Thus, $\cE_{t,k,r}$ describes a set with few $p$-open boundary edges, none of them $q$-open, but many potential connections to infinity.
We first show that thinning from $p$ to $q$ transforms a bad anchored set into 
a configuration satisfying each of these properties at a controlled probability cost.

\begin{lemma}\label{l: thinning}
We have
\[
\mathbb P_p(\cB_{m,r})\le (1-\eps)^{-t}\left(\frac{4p}{p-q}\right)^k\mathbb P(\cE_{t,k,r}).
\]
\end{lemma}

\begin{proof}
Call a $G_p$-connected set $S\subseteq B$ containing $o$ a \emph{$B$-boundary-minimiser} if every proper $G_p$-connected superset $W\subseteq B$ of $S$ satisfies $|\partial_pW|>|\partial_pS|$. 
In other words, $S$ cannot be enlarged within $B$ without increasing its boundary. Since $\Phi_G$ is increasing under inclusion, whenever $\cB_{m,r}$ occurs, it has a $B$-boundary-minimiser witnessing it. 
(Let us note here that this property is the reason to restrict our graph to the large ball $B$.)

To exploit thinning, we count the number of such witnesses that also have many touches and whose boundaries become entirely $q$-closed. We first bound from below their expected number, and then use a standard counting result for small edge cuts to turn this into a probability estimate.

Let $Z$ denote this number, namely
\[
Z=\sum_{\substack{o\in S\subseteq B\\\Phi_G(S)\ge m}}
\mathds{1}_{\{S\text{ is a }B\text{-boundary-minimiser},\,|\partial_pS|\le k\}}
\mathds{1}_{\{\partial_qS=\varnothing\}}
\mathds{1}_{\{t_q(S)\ge t\}}.
\]

Fix a deterministic set $S$ appearing in this sum, and reveal only
the $p$- and $q$-states of the edges with at least one endpoint in $S$.
Conditional on this exposure and the connectivity of $G_p[S]$, the event that $S$ is a
$B$-boundary-minimiser with $|\partial_pS|\le k$ is increasing
in the remaining $p$-edges.
Indeed, suppose that adding edges outside $S$ creates a connected
superset $W\supsetneq S$ whose boundary is no larger than that of $S$.
Let $T$ be the component containing $S$ in the subgraph induced by
$W$ before these edges were added. Since no edge incident to $S$
has changed, $T\supsetneq S$. Moreover, every previously open
edge leaving $T$ also leaves $W$, since otherwise its other
endpoint would belong to $T$. Thus, the original boundary of $T$
is no larger than the new boundary of $W$, contradicting the
boundary-minimality of $S$.
Furthermore, the event $\{t_q(S)\ge t\}$ is increasing in the
$q$-edges outside $S$. Since the event $\{t_q(S)\ge t\}$ depends only on these unrevealed
edges, FKG and Proposition~\ref{prop: touches} give
\begin{equation}\label{eq:touch_bound}
\begin{split}
&\mathbb P\bigl(S\text{ is a }B\text{-boundary-minimiser},\
|\partial_pS|\le k,\ \partial_qS=\varnothing,\
t_q(S)\ge t\bigr)\\
&\qquad\ge (1-\eps)^t
\mathbb P\bigl(S\text{ is a }B\text{-boundary-minimiser},\
|\partial_pS|\le k,\ \partial_qS=\varnothing\bigr).
\end{split}
\end{equation}

Now, reveal the entire $p$-configuration and thin each $p$-open edge independently with retention probability $q/p$. For every $B$-boundary-minimiser with $|\partial_pS|\le k$, the probability that its entire boundary becomes $q$-closed is at least $(1-q/p)^k$. 
This allows us to bound the right-hand side of~\eqref{eq:touch_bound} from below and deduce that
\begin{equation*}
\begin{split}
&\mathbb P\bigl(S\text{ is a }B\text{-boundary-minimiser},\
|\partial_pS|\le k,\ \partial_qS=\varnothing,\
t_q(S)\ge t\bigr)\\
&\qquad\ge (1-\eps)^t (1-q/p)^k
\mathbb P\bigl(S\text{ is a }B\text{-boundary-minimiser},\
|\partial_pS|\le k\bigr).
\end{split}
\end{equation*}
By summing over all such possible sets $S$, and using that $\cB_{m,r}$ guarantees at least one such boundary-minimiser, we obtain $\mathbb EZ\ge(1-\eps)^t(1-q/p)^k\mathbb P_p(\cB_{m,r})$.

It remains to relate $\mathbb EZ$ to the probability of $\cE_{t,k,r}$. Here we use the notion of \emph{important separators} from~\cite{M11}. An important edge separator between a vertex $o$ and a target set is an inclusion-minimal cut which cannot be replaced by a cut of at most the same size that leaves a strictly larger connected component containing $o$. Marx~\cite[Theorem~2]{M11} proved that there are at most $4^k$ important separators of size at most $k$. (See also Remark \ref{rem: recover}, where we recover this result using the algorithm we present in the next subsection).
We apply this result in the finite graph $G_p[B\cup N_G(B)]$, taking $o$ as the `source' and the exterior neighbourhood of $B$ as the `target'. By the maximality property of a $B$-boundary-minimiser, its boundary is precisely such an important separator. 
Consequently, in every $p$-configuration, there are at most $4^k$ possible $B$-boundary-minimisers with boundary at most $k$.

Thus, $Z\le4^k$ pointwise. Furthermore, $Z$ can be positive only if $\cE_{t,k,r}$ occurs. Hence $\mathbb EZ\le4^k\mathbb P(\cE_{t,k,r})$. Combining this with the lower bound on $\mathbb EZ$ proves the lemma.
\end{proof}

We now bound the probability of the same obstruction from the sprinkling viewpoint. The only additional input is Proposition~\ref{prop: feasible}, which counts the possible fingerprints of connected sets with small boundary on a prescribed set of marked vertices. Its proof is postponed to Subsection~\ref{subsec:algo}.

\begin{lemma}\label{l: sprinkling}
There exists a constant $C=C(d,p,q)>1$ such that
\[
\mathbb P(\cE_{t,k,r})\le C\exp\{-(p-q)t/20\}.
\]
\end{lemma}
\begin{proof}
Reveal the entire $q$-configuration, and let $K_q$ be the union
of its infinite clusters. Next, reveal the $p$-states of all
edges with both endpoints in $G\setminus K_q$. Denote by
$\cF$ the configuration revealed so far.
If $o\in K_q$, then $\cE_{t,k,r}$ is impossible: a finite set containing $o$ cannot have $q$-closed boundary. Otherwise, let $H$ be the component of $o$ in the subgraph of $(G\setminus K_q)_p$ induced by $B$, and mark the vertices
$M=\{u\in V(H)\setminus\{o\}:u\text{ has a neighbour in }K_q\}$.

Say that $U\subseteq M$ is \emph{$k$-feasible} if $U=S\cap M$ for some $H$-connected set $S\subseteq V(H)$ containing $o$ with $|\partial_HS|\le k$. For such a set $U$, write $b(U)=|E_G(U\cup\{o\},K_q)|$. Every witness $S$ to $\cE_{t,k,r}$ lies in $H$, since its $q$-boundary is closed. Hence $U=S\cap M$ is $k$-feasible, $b(U)=t_q(S)\ge t$, and at most $k$ of these $b(U)$ edges are $p$-open.

Conditioned on the revealed configuration $\cF$, all edges between $G\setminus K_q$ and $K_q$ are $q$-closed, and they become $p$-open independently with probability $\theta=(p-q)/(1-q)$. Assuming $k\le\theta t/2$, a Chernoff bound and the relation $\theta\ge p-q$ ensure that, for every fixed feasible $U$ with $b(U)\ge t$,
\[
\mathbb P\bigl(\operatorname{Bin}(b(U),\theta)\le k\bigr)
\le\exp(-\theta b(U)/16)\le\exp\{-(p-q)b(U)/16\}.
\]
We will now union-bound over the possible choices of a $k$-feasible set $U$, rather than over all possible witnesses $S$. For this, we use the following counting proposition; its proof is delayed to the next subsection.

\begin{proposition}\label{prop: feasible}
Given a finite graph $H$ with maximum degree at most $d$ containing a vertex $o$ and a vertex set $M\subseteq V(H)\setminus\{o\}$, the number of $k$-feasible vertex subsets of $M$ of size $u$ is at most
\[
\sum_{i=0}^{k}\binom{3k+(d+1)u}{i}.
\]
\end{proposition}

We apply the proposition to the exposed graph $H$ and the marked set $M$. Note that, by excluding $o$ from $M$ and including it automatically in $b(U)$, we avoid any restriction on whether $o$ has neighbours in $K_q$. Notice that $u=|U|\le b(U)\le d(u+1)$. 
Thus, $b(U)\ge t$ implies $u\ge t/d-1$, and $b(U)\ge\max\{t,u\}$ for every relevant $U$. Conditioning on $\cF$, the proposition and a union bound give
\begin{equation}\label{eq:ineq}
\mathbb P(\cE_{t,k,r}\mid\cF)
\le\sum_{u\ge t/d-1}
\left(\sum_{i=0}^{k}\binom{3k+(d+1)u}{i}\right)
\exp\Big\{-(p-q)\frac{\max\{t,u\}}{16}\Big\}.    
\end{equation}
Using that the binomial expression in the term corresponding to $u$ in the above sum is bounded from above by $(3e(du+k)/k)^k$, and choosing $\alpha$ sufficiently small with respect to $c_0$ and $d$ shows that the right-hand side of~\eqref{eq:ineq} is at most
\begin{align*}
\sum_{u\ge t/d-1}
\left(\frac{3e(du+k)}{k}\right)^k
\exp\Big\{-(p-q)\frac{\max\{t,u\}}{16}\Big\}
&\le \sum_{u\ge t/d-1}
\exp\Big\{-(p-q)\frac{\max\{t,u\}}{18}\Big\}\\
&\le Ce^{-(p-q)t/20}.
\end{align*}
for a suitably large constant $C=C(d,p,q)>1$.
Since the bound is uniform in $\cF$, averaging completes the proof.
\end{proof}

Combining the two lemmas, and recalling that $t=c_0m$ and
$k=\alpha m$, we choose $\eps>0$ sufficiently small with respect to $p,q$, and 
$\alpha>0$ sufficiently small compared with $c_0$, so that
\[
\mathbb P_p(\cB_{m,r})
\le C(1-\eps)^{-t}\left(\frac{4p}{p-q}\right)^k
\exp\{-(p-q)t/20\}
\le C\exp\{-\gamma m\}
\]
for some $\gamma=\gamma(G,p,q)>0$, and note that this bound holds uniformly in $r$. 

Finally, we choose $c\le \min\{\gamma,\alpha\}$ in the exponent such that the factor $C$ can be absorbed into it. This is indeed possible: note that, for small $m$, taking the constant $c$ in Theorem~\ref{thm: main} sufficiently small produces a trivial bound on the event that $|C_o|=\infty$.
Letting $r\to\infty$ and taking $m=\Phi_G(n)$ completes the proof.
\qed

\subsection{The algorithm and proof of Proposition~\ref{prop: feasible}}\label{subsec:algo}
In this section, we prove the aforementioned Proposition~\ref{prop: feasible}. 
To prove this statement, we develop an
algorithm that encodes each $k$-feasible set $U\subseteq M$
by a short binary word. Given a connected set $W$
witnessing the feasibility of $U$, the algorithm produces
a binary word from which $U=W\cap M$ can be uniquely recovered.
The key properties of this encoding are that the word
contains at most $k$ zeroes and has length at most
$3k+(d+1)|U|$. Counting such words then gives the
desired bound on the number of feasible sets.

We now describe the algorithm and then establish
these properties. In this description, parameters called \emph{dynamic} are going to evolve during the algorithm; we denote these by bold letters and we use the symbol $\leftarrow$ to update their values.

\begin{algorithm}\label{algo:count}
\textbf{Input}: a finite connected graph $H_0$ containing a vertex $o$ with an ordering $\sigma$ of its edge set $E(H_0)$, an $H_0$-connected vertex set $W\subseteq V(H_0)$ containing $o$, and a set of \emph{marked} vertices $M\subseteq V(H_0)\setminus \{o\}$.
Initiate a dynamic graph $\bH\leftarrow H_0$, dynamic vertex sets $\bC\leftarrow \{o\}$, $\bC'\gets \varnothing$ and $\bT\leftarrow \varnothing$, a dynamic edge set $\bF\gets \varnothing$, and a dynamic integer $\br\gets 0$. 

\textbf{Output}: we construct a binary word $\bw$ initially set to $\varnothing$.

\textbf{Procedure}: at every iteration, first check whether $\bT=W\cap M$. If this is the case, terminate the algorithm and output the word $\bw$.
Otherwise, update $\br$ to be the size of a minimum edge-cutset in $\bH$ between $\bC$ and $M\setminus \bT$.
Next, update $\bC'$ to be the connected vertex set in $\bH$ which:
    \begin{itemize}[noitemsep]
        \item contains $\bC$ but is disjoint from the set $M\setminus \bT$,
        \item satisfies $|\partial_{\bH} \bC'|=\br$,
        \item has maximum size given the previous two conditions, and
        \item if more than one set satisfies the conditions above, $\bC'$ is chosen so that $\partial_{\bH} \bC'$ is minimal with respect to the lexicographic ordering on the edge subsets of $E(H_0)$ induced by $\sigma$.
    \end{itemize}
Set $\bF\leftarrow \partial_{\bH} \bC'$ and let $e=uv$ be the first edge of $\bF$ in the ordering $\sigma$, with $u\in \bC'$.
\begin{enumerate}[\upshape{(\alph*)}]
    \item\label{(a)} If $v\notin W$, append a 0 to $\bw$, update $\bH\leftarrow \bH\setminus \{e\}$ and $\bC\leftarrow \bC'$.
    \item\label{(b)} If $v\in W\setminus M$, append a 1 to $\bw$ and update $\bC\leftarrow \bC'\cup \{v\}$.
    \item\label{(c)} If $v\in W\cap M$, append a 1 to $\bw$ and update $\bC\leftarrow \bC'\cup \{v\}$ and $\bT\leftarrow \bT\cup \{v\}$.
\end{enumerate}
Finally, proceed to the next iteration.
\end{algorithm}

\begin{figure}[htbp]
\centering
\resizebox{\linewidth}{!}{%
\begin{tikzpicture}[x=1cm,y=1cm,font=\small,
  graph edge/.style={draw=black!48,line width=.8pt},
  vertex/.style={circle,fill=black!72,inner sep=1.8pt},
  marked/.style={rectangle,draw=black!70,fill=white,
    line width=.8pt,minimum size=5.5pt,inner sep=0pt},
  source region/.style={draw=statecolor,fill=statecolor!8,
    line width=1pt}]
  \definecolor{statecolor}{RGB}{38,113,112}
  \definecolor{cutcolor}{RGB}{183,62,62}
  \definecolor{purplecolor}{RGB}{117,89,145}
  \definecolor{edgecolor}{RGB}{168,112,43}
  \path[use as bounding box] (-.15,.1) rectangle (14.5,4.55);

  \foreach \panel/\offset in {0/0,1/7.6}{
    \begin{scope}[xshift=\offset cm]
      \coordinate (o) at (.85,2.65);
      \coordinate (a) at (2.0,3.60);
      \coordinate (b) at (2.0,1.85);
      \coordinate (u) at (3.15,2.65);
      \coordinate (v) at (4.55,2.65);
      \coordinate (mone) at (6.10,3.60);
      \coordinate (mtwo) at (6.10,1.40);

      \path[draw=black!40,fill=black!3,line width=.8pt]
        (.05,2.65)
        .. controls (.05,3.50) and (1.05,4.45) .. (2.0,4.45)
        -- (6.05,4.45)
        .. controls (6.65,4.45) and (6.75,4.05) .. (6.75,3.60)
        .. controls (6.75,3.15) and (6.35,3.00) .. (5.95,3.00)
        .. controls (5.45,3.00) and (5.25,1.95) .. (4.55,1.95)
        -- (4.05,1.95)
        .. controls (3.30,1.95) and (2.85,.85) .. (2.0,.85)
        .. controls (1.05,.85) and (.05,1.80) .. cycle;
      \node[text=black!60] at (5.35,4.13) {$W$};

      \ifnum\panel=0
        \tikzset{cut region/.style={
          draw=cutcolor,fill=cutcolor!8,line width=1pt}}
      \else
        \tikzset{cut region/.style={
          draw=purplecolor,dashed,line width=1pt}}

        \path[source region]
          (.25,2.65)
          .. controls (.25,3.15) and (1.20,4.23) .. (2.0,4.23)
          .. controls (2.80,4.23) and (3.55,3.16) .. (4.03,3.16)
          -- (4.54,3.16)
          .. controls (4.96,3.16) and (5.05,2.90) .. (5.05,2.65)
          .. controls (5.05,2.40) and (4.96,2.14) .. (4.54,2.14)
          -- (4.03,2.14)
          .. controls (3.55,2.14) and (2.80,1.07) .. (2.0,1.07)
          .. controls (1.20,1.07) and (.25,2.15) .. cycle;
      \fi

      \path[cut region]
        (.43,2.65)
        .. controls (.43,3.10) and (1.35,4.02) .. (2.0,4.02)
        .. controls (2.65,4.02) and (3.57,3.10) .. (3.57,2.65)
        .. controls (3.57,2.20) and (2.65,1.28) .. (2.0,1.28)
        .. controls (1.35,1.28) and (.43,2.20) .. cycle;

      \ifnum\panel=0
        \path[source region] (o) circle[radius=.36];
      \fi

      \draw[graph edge] (o)--(a)--(u)--(b)--(o);
      \draw[graph edge] (v)--(mone) (v)--(mtwo);
      \draw[draw=edgecolor,line width=1.5pt] (u)--(v);
      \node[above=4pt,text=edgecolor] at (3.85,2.65) {$e$};

      \foreach \name in {o,a,b,u,v}
        \node[vertex] at (\name) {};
      \node[marked] at (mone) {};
      \node[marked] at (mtwo) {};

      \node[below=1pt] at (o) {$o$};
      \node[above=1pt] at (a) {$a$};
      \node[below=1pt] at (b) {$b$};
      \node[below=4pt] at (u) {$u$};
      \node[below=4pt] at (v) {$v$};
      \node[above=3pt] at (mone) {$m_1$};
      \node[right=5pt] at (mtwo) {$m_2$};

      \ifnum\panel=0
        \node at (3.4,.40)
          {$\bC=\{o\},\quad
            \bC'=\{o,a,b,u\},\quad
            \bw=\varnothing$};
      \else
        \node at (3.4,.40)
          {$\bC\gets \bC'\cup\{v\},
            \qquad \bw\gets 1$};
      \fi
    \end{scope}
  }
\end{tikzpicture}%
}
\caption{Initial state (left) and state after one iteration (right)
of Algorithm~\ref{algo:count}.
The fixed witness $W=\{o,a,b,u,v,m_1\}$ is shaded grey;
squares indicate the marked vertices $M=\{m_1,m_2\}$. The boundary of $\bC'$ on the left is $\bF=\{e\}$.
Since $v\in W\setminus M$, case~\ref{(b)} gives
$\bC\gets \bC'\cup\{v\}$ (coloured in teal) and $\bw$ becomes $1$.}
\label{fig:encoding-algorithm}
\end{figure}
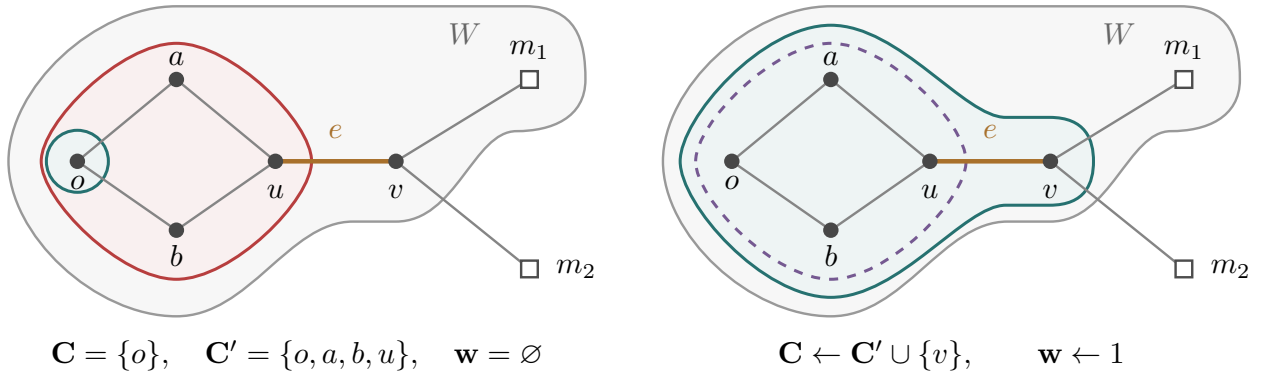
\FloatBarrier

Observe that the algorithm is well-defined and terminates. Indeed, only edges with an endpoint outside $W$ are deleted, so any vertex $v\in (W\cap M)\setminus \bT$ remains connected to $o$ in $\bH$; hence, as long as the algorithm is still running, the next edge $e$ always exists. Moreover, each iteration either deletes an edge or enlarges $\bC$.
Let us note here that the choice of a largest boundary-minimising set $\bC'$ is crucial for controlling the length of the word: case~\ref{(b)} forces the minimum-cut size to increase, while the possible decreases in cases~\ref{(a)} and~\ref{(c)} will be shown to be controlled by $k$ and $|U|$, respectively.

We first establish that the output word determines $W\cap M$.

\begin{lemma}\label{cl: property 1}
Given the input $H_0,M,o,\sigma$ and the output $\bw$ of the algorithm upon termination, the set $W\cap M$ can be uniquely determined.
\end{lemma}
\begin{proof}
Starting from the known initial state, we reconstruct the algorithm one bit at a time. The current values of $\bH,\bC,\bT$, together with $\sigma$, determine $\bC'$, $\bF$, and the next edge $e=uv$. If the next bit is 0, we follow case~\ref{(a)}. If it is 1, we follow case~\ref{(b)} or~\ref{(c)}, according to whether $v\notin M$ or $v\in M$. 
Thus, each bit uniquely determines how to update each of $\bH,\bC,\bT$.
After reading the entire word, the recovered set $\bT$ equals $W\cap M$, as required.
\end{proof}

\noindent
We emphasise that the purpose of the algorithm and the decoding Lemma~\ref{cl: property 1} is not to recover
the set $W$ itself, but only the set $U=W\cap M$. Indeed, many different sets $W$ may have the same intersection
with $M$, and distinguishing between them would lead to a substantial overcounting. 
The key feature of our algorithm is precisely that it encodes only the information needed to recover $U$. Thus, for fixed $H_0,M,o,\sigma$, each output word determines at most one subset of $M$.

Recall that a set $U\subseteq M$ is $k$-feasible if
there exists an $H_0$-connected set $W$ containing $o$
such that $W\cap M=U$ and $|\partial_{H_0}W|\le k$.
Call any such $W$ a witness for $U$. We can now prove Proposition~\ref{prop: feasible}.
\begin{proof}[Proof of Proposition~\ref{prop: feasible}]
Fix an arbitrary ordering $\sigma$ of the edges of $H_0$.
Given a $k$-feasible set $U\subseteq M$ of size $t$, choose an arbitrary witness $W$ and let $w$ be the binary word produced by Algorithm~\ref{algo:count}.
By Lemma~\ref{cl: property 1}, different feasible sets $U$ give different words. 
Thus, it suffices to count the possible output words.

First, we show that $w$ contains at most $k$ zeroes.
Crucially, note that an edge deleted when writing a zero need not belong to $\partial_{H_0}W$ since the boundary-minimising set $\bC'$
may contain vertices outside $W$; we therefore track the boundary of an enlarged witness $W\cup \bC'$.
Let $H_{i-1},C_{i-1},T_{i-1}$ denote the values of $\bH,\bC,\bT$ before iteration $i$, and let $C_i',F_i$ be the values of $\bC',\bF$ during that iteration.
Set $W_0=W$ and $W_i=W\cup C_i'$ for $i\ge1$ (the $W_i$ will play the role of the enlarged witness that we track).
As the sets $C_i'$ are nested, we have $W_i=W_{i-1}\cup C_i'$.
Also, $C_{i-1}\subseteq W_{i-1}$ (since $C_{i-1}\setminus C_{i-1}'$ is either a single vertex in $W$ or empty),
while $C_i'$ is disjoint from $M\setminus T_{i-1}$. Hence, $\partial_{H_{i-1}}(W_{i-1}\cap C_i')$ 
separates $C_{i-1}$ from $M\setminus T_{i-1}$, and the minimality of $F_i$ gives $|\partial_{H_{i-1}}(W_{i-1}\cap C_i')|\ge |F_i|$.
By submodularity of the edge boundary (see e.g.\ \cite[Section~1.1]{KG14}),
\begin{align*}
|\partial_{H_i}W_i|\le |\partial_{H_{i-1}}W_i|
&\le |\partial_{H_{i-1}}W_{i-1}|+|F_i|-|\partial_{H_{i-1}}(W_{i-1}\cap C_i')|\\
&\le |\partial_{H_{i-1}}W_{i-1}|.
\end{align*}
The first inequality is strict in case~\ref{(a)}, since the deleted edge $e=uv$ crosses the boundary of $W_i$: indeed, $u\in C_i'\subseteq W_i$ and $v\notin W\cup C_i'=W_i$.
Thus, the boundary of the enlarged witness $W_i$ never increases and decreases by at least one whenever we write a zero. Since $|\partial_{H_0}W_0|\le k$, the word has at most $k$ zeroes.

Next, we bound the length of $w$. If $t=0$, there is only one possible set $U=\varnothing$, so the proposition is immediate. Thus assume $t\ge1$.
In case~\ref{(a)}, deleting one edge decreases the minimum-cut value by at most one, so $|F_{i+1}|-|F_i|\ge-1$.
In case~\ref{(c)}, we have $|F_{i+1}|-|F_i|\ge-d$: adding the at most $d$ edges incident to the newly recorded marked vertex $v$ to any cut for the next step gives a cut for the preceding step.
Crucially, in case~\ref{(b)}, $|F_{i+1}|-|F_i|\ge1$. Indeed, the graph $H_i$ and target set $M\setminus T_i$ are unchanged, so the minimum-cut value cannot decrease; if it stayed unchanged, there would be 
an $H_i$-connected set with minimum boundary containing $C_i'\cup\{v\}$, contradicting the maximality of $C_i'$.

Let $\tau$ be the number of iterations, so that $T_\tau=U$.
Recall that the submodularity argument gives $|\partial_{H_i}W_i|\le k$ throughout the algorithm.
At intermediate stages, the boundary of $W_i$, together with the edges incident to the vertices in $U\setminus (T_i\cup C_i)$, separates $C_i$ from $M\setminus T_i$. At termination, however, we have that $U\subseteq C_\tau$. Since $C_\tau\subseteq W_\tau$ and $W_\tau\cap M=U$, the boundary of $W_\tau$ itself separates $C_\tau$ from $M\setminus U$. Thus, letting $F_{\tau+1}$ be a minimum cut between these sets in $H_\tau$, we obtain
\begin{equation*}
|F_{\tau+1}|
\le |\partial_{H_\tau}W_\tau|
\le |\partial_{H_0}W|\le k.
\end{equation*}
Writing $\alpha,\beta,\gamma$ for the numbers of iterations of types~\ref{(a)},~\ref{(b)},~\ref{(c)}, respectively, we have $\alpha\le k$ and $\gamma=t$.
The preceding bounds imply $\beta-\alpha-d\gamma\le |F_{\tau+1}|-|F_1|\le k$, and consequently $\tau=\alpha+\beta+\gamma\le 3k+(d+1)t$.

Finally, pad every output word with 1-bits to length $L=3k+(d+1)t$. Since we are counting feasible sets $U$ of fixed size $t$, the padded word still determines $U$: we run the same procedure described above until exactly $t$ marked vertices have been recovered. Thus, padding preserves injectivity, and the number of possible words is at most
\[
\sum_{i=0}^{k}\binom{3k+(d+1)t}{i},
\]
as desired.
\end{proof}

\begin{remark}\label{rem: recover}
Our algorithm also allows us to recover the bound on
important separators from~\cite[Theorem~2]{M11}, used
earlier in the proof.

Let $H$ be a locally finite graph containing a vertex $o$,
and let $B\ni o$ be a finite vertex set.
Recall that a connected set $W\subseteq B$ containing $o$ is a
$B$-minimiser if there is no proper connected superset
$S\supsetneq W$, $S\subseteq B$, with
$|\partial_H S|\le |\partial_H W|$.
We apply Algorithm~\ref{algo:count} with $H_0$ being
the component of $o$ in $H[B\cup N_H(B)]$ and
$M=N_H(B)\cap V(H_0)$, replacing the stopping
condition by the requirement that the minimum cut-value
becomes zero.

For a $B$-boundary-minimiser $W$, induction shows that every
value of $\bC'$ throughout the algorithm is contained
in $W$, and every deleted edge belongs to $\partial_H W$.
Indeed, assuming this for previously deleted edges,
the minimum-cut property and submodularity imply that
adding the chosen 
boundary-minimiser to $W$ cannot
increase its boundary. By $B$-boundary-minimality, this side
must therefore be contained in $W$, and any subsequently
deleted edge belongs to $\partial_H W$.
Consequently, upon termination, the component of $o$
in the remaining graph contains $W$, is contained in
$B$, and has all its boundary edges among those deleted.
By $B$-boundary-minimality, this component must equal $W$,
which can therefore be recovered from the output word.

Moreover, since $W\cap M=\varnothing$, case~\ref{(c)}
never occurs. If $|\partial_H W|\le k$, the word
contains at most $k$ zeroes, since each zero deletes
a distinct edge of $\partial_H W$. Each one increases
the minimum cut-value by at least one, while each
zero decreases it by at most one. Since the final
cut-value is zero, the number of ones is at most
the number of zeroes. Thus, the word has length at
most $2k$. As before, padding shorter words with ones preserves
injectivity. 
This gives at most
$\sum_{i=0}^k\binom{2k}{i}\le 4^k$ such minimisers.
\end{remark}

\vspace{0.2cm}
\noindent
\textbf{Acknowledgements} The authors are grateful to Vincent Tassion and Philip Easo for fruitful discussions. BS was supported in part by SNSF grant 20021-228014.

\textbf{Use of AI:} AI models were used only to generate the figures from photos of hand-drawn ones, as well as to assist with proofreading of the paper.

\bibliographystyle{alpha}
\bibliography{Bib}

\end{document}